\documentclass[10pt]{amsart}
\usepackage{appendix}
\usepackage{graphicx,enumitem,dsfont}
\usepackage[usenames]{color}
\usepackage[colorlinks]{hyperref}
\usepackage{subfigure}

\newcommand{\Z}{{\mathbb{Z}} }
\newcommand{\N}{{\mathbb{N}}}
\newcommand{\R}{\mathbb{R}}

\newcommand{\sech}{\mathrm{sech}}

\newcommand{\Comment}[1]{}

\renewcommand{\i}{\ifmmode\mathit{\mathchar"7010 }\else\char"10 \fi}
\renewcommand{\j}{\ifmmode\mathit{\mathchar"7011 }\else\char"11 \fi}

\newcommand{\Dx}{\Delta x}

\newcommand{\Dt}{\Delta t}

\newcommand{\norm}[1]{\left\|#1\right\|}

\newcommand{\Dpx}{D_+}

\newcommand{\Dmx}{D_-}

\newcommand{\Dcx}{D_0}

\newcommand{\Dpt}{D_t^+}

\newtheorem{theorem}{Theorem}[section]

\newtheorem{lemma}{Lemma}[section]

\newtheorem{remark}{Remark}[section]
\theoremstyle{definition} 

\newtheorem*{maintheorem*}{Main Theorem}

\allowdisplaybreaks

\numberwithin{equation}{section}
\numberwithin{figure}{section}
\numberwithin{table}{section}

\newcounter{asnr}

{\ifnum\value{asnr}=0 \stepcounter{asnr} 
  \begin{enumerate}[label=\textbf{A}.\arabic{enumi}]
    \else
    \begin{enumerate}[label=\textbf{A}.\arabic{enumi},resume] \fi}
{\end{enumerate}}

\title[Kawahara equation]{Convergence of numerical schemes for the Korteweg-de Vries-Kawahara Equation.}

\author[U. Koley]{U. Koley} \address[Ujjwal Koley]{\newline Centre of Mathematics for Applications (CMA)
  \newline University of Oslo\newline P.O. Box 1053, Blindern\newline
  N--0316 Oslo, Norway} \email[]{ujjwalk@cma.uio.no}

\date{\today}

\begin{document}

\begin{abstract}
  We are concerned with the convergence of a numerical scheme for the initial-boundary value problem associated to the Korteweg-de Vries-Kawahara equation (in short Kawahara equation), which is a transport equation perturbed by dispersive terms of $3$rd and $5$th order. This equation appears in several fluid dynamics problems. It describes the evolution of small but finite amplitude long waves in various problems in fluid dynamics. We prove here the convergence of both semi-discrete as well as fully-discrete finite difference schemes for the Kawahara equation. Finally, the convergence is illustratred by several examples.
\end{abstract}

\maketitle

\section {Introduction}
This paper is concerned with the initial-boundary value problem of the Kawahara equation:
\begin{equation}
\begin{aligned}
\label{eq:kawahara}
u_t & = - u u_x - u_{xxx} + u_{xxxxx},
\end{aligned}
\end{equation}
with initial condition
\begin{equation}
\begin{aligned}
\label{eq:kawahara_initial}
u (x,0) &= f(x),\quad \text{for all $x$}
\end{aligned}
\end{equation}
and the boundary condition
\begin{equation}
\begin{aligned}
\label{eq:kawahara_boundary}
\qquad \qquad u(x,t)& = u(x+1,t), \quad \text{for all $x$ and $t$}
\end{aligned}
\end{equation}
It is well known that the one-dimensional waves of small but finite amplitude in dispersive systems (e.g., the magneto-acoustic waves in plasmas, the shallow water waves, the lattice waves and so on) can be described by the Korteweg-de Vries (KdV in short) equation, given by
\begin{equation}
\begin{aligned}
\label{eq:kdv}
u_t  = - u u_x - u_{xxx} , 
\end{aligned}
\end{equation}
which admits either compressive or rarefactive steady solitary wave solution (by a solitary water wave, we mean a travelling wave solution of the water wave equations for which the free surface approaches a constant height as $|x| \rightarrow \infty$) according to the sign of the dispersion term (the third order derivative term). Under certain circumstances, however, it might happen that the coefficient of the third order derivative in the KdV equation becomes small or even zero. In that case one has to take account of the higher order effect of dispersion in order to balance the nonlinear effect. In such cases one may obtain a generalized nonlinear dispersive equation, known as Kawahara equation, which has a form of the KdV equation with an additional fifth order derivative term given by \eqref{eq:kawahara}.
The Kawahara equation is an important nonlinear dispersive equation. It describes solitary wave propagation in media in which the first-order dispersion is anomalously small. A more specific physical background of this equation was introduced by Hunter and Scheurle \cite{hun}, where they used it to describe the evolution of solitary waves in fluids in which the Bond number is less than but close to $\frac{1}{3}$ and the Froude number is close to $1$. In the literature this equation is also referred to as the fifth order KdV equation or singularly perturbed KdV equation. The fifth order term $\partial_x^5 u$ is called the Kawahara term.
There has been a great deal of work on solitary wave solutions of the Kawahara equation \cite{kawa,kic,kato1,pon,gleb} over the past thirty years. It is found that, similarly to the KdV equation, the Kawahara equation also has solitary wave solutions which decay rapidly to zero as $t\to \infty$, but unlike the KdV equation whose solitary wave solutions are non-oscillating, the solitary wave solutions of the Kawahara equation have oscillatory trails. This shows that the Kawahara equation is not only similar but also different from the KdV equation in the properties of solutions, like what happens between the formulations of this equation and the KdV equation. The strong physical background of the Kawahara equation and such similarities and differences between it and the KdV equation in both the form and the behavior of the solution render the mathematical treatment of this equation particularly interesting. The Cauchy problem given by \eqref{eq:kawahara}, \eqref{eq:kawahara_initial} has been studied by a few authors \cite{bia,ver,kato2,hua,shang}. It has been shown that the problem \eqref{eq:kawahara}, \eqref{eq:kawahara_initial} has a local solution $u \in C([-T,T];H^r(\R))$ if $f \in H^r(\R)$ and $r>-1$. This local result combined with the energy conservation law yields that \eqref{eq:kawahara} has a global solution $u \in C([-\infty,\infty];L^2(\R))$ if $f \in L^2(\R)$. Well-posedness results can be found in \cite{bia}.

In this paper, we focus on the the derivation of convergent finite difference numerical methods for the initial-boundary value problem \eqref{eq:kawahara}, \eqref{eq:kawahara_initial}, \eqref{eq:kawahara_boundary}. The problem of analyzing convergent numerical schemes of course intimately connected with the mathematical properties of the Cauchy problem for the KdV equation, which is well developed in literature. As far as we are concerned, the first mathematical proof of the existence and uniqueness of the solutions of the KdV equation \eqref{eq:kdv} was accomplished by Sj\"{o}berg \cite{sjo}
in $1970$, using a finite difference semi-discrete approximation. In \cite{nils}, authors have considered a fully-discrete finite difference scheme for the KdV equation and showed the convergence of their scheme. In this paper, we are going to use a similar technique here. Note that Sj\"{o}berg's approach is based on a semi-discrete approximation where one discretizes the spatial variable, thereby reducing the equation to a system of ordinary differential equations. However, we also stress that for numerical computaions this set of ordinary differential equations should be further discretized in order to be solved. Thus in order to have a completely satisfactory numerical method, one seeks a fully discrete scheme that reduces the actual computation to a solution of a finite set of algebric equations. In the present paper, we consider both semi-discrete and fully-discrete cases to prove the following theorem:
\begin{theorem}
\label{theo:maintheorem}
If $f(x)$ is a $1$-periodic function and fifth derivative of $f(x)$ belongs to $L^2(\R)$, then there exists a unique solution of the problem \eqref{eq:kawahara},\eqref{eq:kawahara_initial} and \eqref{eq:kawahara_boundary}, i.e., there exists a function $u \in H^5(\R)$ with $u_t \in L^2(\R)$ for which the euality \eqref{eq:kawahara} holds as an $L^2$ equality.
\end{theorem} 
The numerical computaion of solutions of the Kawahara equation is rather capricious. Two competing equations are involved, namely nonlinear convective term $uu_x$, which in the context of the equation $u_t = uu_x$ yields an infinite gradient in finite time even for smooth data, and the linear dispersive terms $u_{xxx}, u_{xxxxx}$, which in the context of the equation $u_t = u_{xxx} + u_{xxxxx}$ produces hard to compute dispersive waves, and these two effects combined makes it difficult to obtain accurate and fast numerical methods. Most of the finite difference schemes will consist of a sum of two terms, one discretizing the convective term and one discretizing the dispersive terms. These two effects will have to balance each other, as it is known that the Kawahara equation itself keeps the Sobolev norm $H^s (s>-1)$ bounded.

The rest of the paper is organized as follows: In section $2$, we consider the semi-discrete scheme for the initial-boundary value problem corresponding to \eqref{eq:kawahara}. At first, we state some of the well-known Sobolev type estimates and then we showed the convergence of the semi-discrete scheme. Both local and global existence has been proved. In section $3$, we consider a fully-discrete semi-implicit scheme for the initial-boundary value problem corresponding to \eqref{eq:kawahara}. We have used explicit discretization for the ``nonlinear'' term and implicit discretizations for the ``dispersive'' terms. Convergence of the fully-discrete scheme has been shown in this section. In section $4$, we have shown the uniqueness of the solution to the initial-boundary value problem given by \eqref{eq:kawahara}. Finally, in section $5$, we have justified the convergence results by several numerical experiments, namely Soliton experiments. We also compare our scheme, based on the fully-discrete semi-implicit scheme, with the existing results in literature.

\section{Semi Discrete Analysis}
\subsubsection{Local existence:}

Let $\Dx$ be a small positive number and define a grid on the X-axis to be the set of gridpoints $x_i=i\Dx$ for $i \in \Z$. We will denote the value of our difference approximation at $x_i$ as $u(x_i,t)= u_i$. Since we are in the periodic case, for simplicity we assume a unit period and that $1/\Dx \in \N$. In that case $u_{i+N} = u_i$ for $i \in \Z$. To simplify the notations, we will introduce the finite difference operators:
\begin{align*}
\Dmx u_i = \frac{1}{\Dx}(u_i - u_{i-1}), \quad \Dpx u_i = \frac{1}{\Dx}(u_{i+1} - u_{i}), \quad \Dcx u_i = \frac{1}{2\Dx}(u_{i+1} - u_{i-1}).
\end{align*}
We will also use the following notations:
\begin{equation*}
\begin{aligned}
\norm {f}^2 = (f,f) \quad \text{and} \quad (f,g)= \int_{0}^{1} \overline {f(x)} g(x)\,dx,
\end{aligned}
\end{equation*}
and in the space of gridfunctions ( a discrete, possibly complex valued, function defined on the grid ), we define the scalar product and the norm by
\begin{align*}
(f,g)_h = h\sum_{i} \overline {f(x_i)} g(x_i) \quad \text {and $\norm {f}^2_{h} = (f,f)_h$}.
\end{align*}
Later we need the following Lemmas, proven in \cite{sjo}. We begin with a well-known Sobolev-type Lemma:
\begin{lemma}
\label{lemma:Lemma-1}
Let $\sigma$ and $\tau$ be integers such that $0\le\tau < \sigma$. Then for every constant $\epsilon>0$ there exists a constant $c(\epsilon)$ such that for all functions $y$, sufficiently differentiable on $0 \le x \le 1$,
\begin{equation}
\label{eq:est1}
\begin{aligned}
\max_{0\le x\le 1} |\frac{{\partial}^{\tau} y}{{\partial} x^{\tau}}|^2 \le \epsilon \norm {\frac{{\partial}^{\sigma} y}{\partial x^{\sigma}}}^2 + c(\epsilon) \norm{y}^2 ,
\end{aligned}
\end{equation}
and 
\begin{equation}
\label{eq:est2}
\begin{aligned}
\norm {\frac{{\partial}^{\tau} y}{{\partial} x^{\tau}}}^2 \le \epsilon \norm {\frac{{\partial}^{\sigma} y}{\partial x^{\sigma}}}^2 + c(\epsilon) \norm{y}^2 .
\end{aligned}
\end{equation}
\end{lemma}

\begin{lemma}
\label{lemma:lemma-2}
Let $\tau_1$ and $\tau_2$ be nonnegative integers with $\tau_1 + \tau_2 = \tau$ and $\psi$ be a function of the form
\begin{align*}
\psi = \sum_{\omega \in \Z} a(\omega) e^{2\pi i \omega x}.
\end{align*}
Then 
\begin{align*}
\left (\frac{2}{\pi} \right )^{2\tau} \norm {\frac{{\partial}^{\tau} \psi}{{\partial} x^{\tau}}}^2 \le \norm {\Dpx^{\tau_1}\Dmx^{\tau_2} \psi}^2
= \norm {\Dpx^{\tau_1}\Dmx^{\tau_2} \psi}^2_{h} \le \norm {\frac{{\partial}^{\tau} \psi}{{\partial} x^{\tau}}}^2.
\end{align*}
\end{lemma}

\begin{lemma}
\label{lemma:lemma-3}
Let $\sigma$ and $\tau$ be integers such that $0\le\tau < \sigma$, and $z$ be any gridfunction. Then for every constant $\epsilon>0$ there exists a constant $c(\epsilon)$ independent of $h$ and $z$ such that,
\begin{equation}
\label{eq:est3}
\begin{aligned}
\norm {\Dpx^{\tau} z}^2_{h} \le \epsilon \norm {\Dpx^{\sigma} z}^2_{h} + c(\epsilon) \norm{z}^2_{h}.
\end{aligned}
\end{equation}
\end{lemma}

\begin{remark}
\label{remark-1}
The inequality \eqref{eq:est3} can be modified as follows: One can replace $\Dpx^{\tau}$ by any operator $\Dpx^{\tau_1}\Dcx^{\tau_2}\Dmx^{\tau_3}$ with $\tau =\tau_1 +\tau_2 +\tau_3$. The right-hand side can likewise be replaced by $ \Dpx^{\sigma_1}\Dmx^{\sigma_2}$ with $\sigma = \sigma_1 + \sigma_2$. However, it can be proven by counterexamples that the lemma is not true if the right member contains $\Dcx$.
\end{remark}

\begin{remark}
\label{remark-2}
From the above Lemmas the folloing result follows immediately
\begin{equation}
\label{eq:est4}
\begin{aligned}
\max_{i} |\Dcx z_i|^2 \le \epsilon \norm{ \Dpx^3 \Dmx^2 z}^2_{h} + c(\epsilon) \norm{z}^2_{h},
\end{aligned}
\end{equation}
where $\epsilon$ and $c(\epsilon)$ have the same properties as before.
\end{remark}

Now we are in a position to state the semi-discrete scheme of the Kawahara equation \eqref{eq:kawahara}, given by
\begin{equation}
\begin{aligned}
\label{eq:discrete}
(u_i)_t & = - \frac{1}{3}[ u_i \Dcx u_i  + \Dcx u_i^2] - \Dmx \Dpx^{2} u_i + \Dpx^3 \Dmx^2 u_i, \quad \text{$ i= 1,2,...,N$}, 
\end{aligned}
\end{equation}
with the initial condition 
\begin{equation}
\begin{aligned}
\label{eq:discrete_initial}
u_i(0)= f(x_i), \quad \text{$ i= 1,2,...,N$}
\end{aligned}
\end{equation}
and the boundary condition
\begin{equation}
\begin{aligned}
\label{eq:discrete_boundary}
u_i(t)= u_{i+N}(t) \quad \text{ for all $i$ and $t$}
\end{aligned}.
\end{equation}

We will first prove the local existence of a solution for $t>0$. The case $t<0$ will be treated later.
 
\begin{theorem}
\label{theo:localexistence}
There exists a time $T_1>0$ and constants $k_i$, $i= 0,1,2,3,4$ independent of $\Dx$ but dependent on $f(x)$ and its derivatives of order five and lower, such that
\begin{equation}
\label{eq:exis1}
\begin{aligned}
\norm {u(\cdot,t)}_{h} \le k_0, \qquad \text {for all $t$}
\end{aligned}
\end{equation}
\begin{equation}
\label{eq:exis2}
\begin{aligned}
\qquad \qquad \qquad \qquad | u(x_i, t)|       \le k_1, \qquad  \text {$0 \le t \le T_1$}, \qquad \text {for all $i$}
\end{aligned}
\end{equation} 
\begin{equation}
\label{eq:exis3}
\begin{aligned}
\norm {\Dmx \Dpx^{2} u(\cdot,t)}_{h}  \le k_2, \qquad \text {$0 \le t \le T_1$}
\end{aligned}
\end{equation} 
\begin{equation}
\label{eq:exis4}
\begin{aligned}
\norm {\Dpx^{3} \Dmx^{2} u(\cdot,t)}_{h}  \le k_3, \qquad \text {$0 \le t \le T_1$}
\end{aligned}
\end{equation} 
and finally with,
\begin{equation*}
\begin{aligned}
v(x,t) &= \frac{\partial u(x,t)}{\partial t} 
\end{aligned}
\end{equation*}
\begin{equation}
\label{eq:exis5}
\begin{aligned}
\qquad \qquad \norm {v(\cdot,t)}_{h}  \le k_4,\qquad \text {$0 \le t \le T_1$}.
\end{aligned}
\end{equation} 
\end{theorem}

\begin{proof}
Multiply the above equation \eqref{eq:discrete} by $\Dx u_i$ and summing over all $i$, we have
\begin{align*}
(u,u_t)_h & = -\frac{1}{3} [ (u, u \Dcx u)_h  + (u, \Dcx u^2)_h ] - (u,  \Dmx \Dpx^2 u)_h + (u,  \Dpx^3 \Dmx^2 u)_h \\
& = - \frac{\Dx}{2} \norm {\Dpx \Dmx u}^2_h - \frac{\Dx}{2} \norm {\Dmx \Dpx^2 u}^2_h
\end{align*}
where we are using the following identities:
\begin{equation}
\label{eq:ineq}
\begin{aligned}
 (u,  \Dmx \Dpx^2 u)_h & = \frac{1}{2}(u, \Dmx \Dpx^2 u)_h - \frac{1}{2}(u, \Dpx \Dmx^2 u)_h  \\
& = \frac{1}{2} \left( u, \Dpx \Dmx ( \Dpx -\Dmx)u \right)_h\\
& = \frac{\Dx}{2} \norm {\Dpx \Dmx u}^2_h ,
\end{aligned}
\end{equation}
since $ (u, \Dpx \Dmx^2u)_h = - (u, \Dmx \Dpx^2u)_h$, because $(u, \Dpx u)_h= -(u, \Dmx u)_h$, in the first line, and 
\begin{align*}
\Dpx -\Dmx = \Dx \Dmx \Dpx = \Dx \Dpx \Dmx.
\end{align*}
In a similar manner, we find that
\begin{align*}
(u, \Dpx^3 \Dmx^2 u)_h & = \frac{1}{2}(u, \Dpx^3 \Dmx^2 u)_h - \frac{1}{2}(u, \Dmx^3 \Dpx^2 u)_h  \\
& = \frac{1}{2} \left( u, \Dmx^2 \Dpx^2 ( \Dpx -\Dmx)u \right)_h\\
& = - \frac{\Dx}{2} \norm {\Dmx \Dpx^2 u}^2_h .
\end{align*} 
So, finally we conclude that
\begin{align*}
\norm {u(\cdot, t)}^2_h  & + \frac{\Dx}{2} \int_{0}^{t}\left(\norm {\Dpx \Dmx u(\tau)}^2_h  + \norm {\Dmx \Dpx^2 u(\tau)}^2_h \right) \,d\tau \\
& \le \norm {u(\cdot, 0)}^2_h  = \norm{f}^2_h \le 2\int_{0}^{1} f^2(x)\,dx = k_0^2, 
\end{align*}
from which \eqref{eq:exis1} follows.

Keeping in mind that $ v = u_t $, from the equation \eqref{eq:discrete} using the triangle inequality and Lemma $2.3$, we get the following inequality
\begin{align*}
 \norm {\Dpx^3 \Dmx^2 u}_h  & \leq \norm v_h + \frac{1}{3} \norm{[ u \Dcx u + \Dcx u^2 ]}_h + \norm {\Dmx \Dpx^2 u}_h\\
& \le \norm v_h + \max |\Dcx u| \norm u_h + ( \epsilon \norm {\Dpx^3 \Dmx^2 u}_h + c(\epsilon) \norm u_h)\\
& \le \norm v_h + \norm u_h ( \epsilon \norm {\Dpx^3 \Dmx^2 u}_h + c(\epsilon) \norm u_h ) \\
& \qquad \qquad +( \epsilon \norm {\Dpx^3 \Dmx^2 u}_h + c(\epsilon) \norm u_h )
\end{align*}
Now we can use the fact that $\norm u_h \le k_0$, and choose $\epsilon$ such that 
\begin{equation}
\label{eq:fifthorder}
\begin{aligned}
\norm {\Dpx^3 \Dmx^2 u}_h \le \nu_1 \norm v_h + \nu_2 ,
\end{aligned}
\end{equation}
where $\nu_1$ and $\nu_2$ are constants independent of $\Dx$.
To get a bound for $\norm v_h$, we will proceed as follows. After differentiating both sides of the equation \eqref{eq:discrete} with respect to $t$, we get the following equation
\begin{equation}
\label{eq:discrete1}
\begin{aligned}
(v_i)_t & = - \frac{1}{3} [ v_i \Dcx u_i + u_i \Dcx v_i + 2 \Dcx (u_iv_i)] - \Dmx \Dpx^2  v_i  +  \Dpx^3 \Dmx^2  v_i .
\end{aligned}
\end{equation}
Now multiplying equation \eqref{eq:discrete1} by $hv_i$ and summing over $i$, we have
\begin{equation*}
\begin{aligned}
(v, v_t)_h & = - \frac{1}{3} [ (v ,v \Dcx u)_h + (v, u \Dcx v)_h + 2 ( v, \Dcx uv )_h]\\
& \qquad \qquad - ( v, \Dmx \Dpx^2  v)_h  +   ( v, \Dpx^3 \Dmx^2  v)_h\\
& \le \frac{1}{3}| (v^2, \Dcx u)_h + ( v, \Dcx uv )_h| - ( v, \Dmx \Dpx^2  v)_h  + ( v, \Dpx^3 \Dmx^2  v)_h\\
& \le \frac{1}{3}| (v^2, \Dcx u)_h + \frac{\Dx}{2}( v, \Dpx u \Dcx v )_h + \frac{1}{2}( v, v_{j-1}\Dpx u )_h| \\
& \le C \left [ \max |\Dcx u| \norm v^2_h + \max |\Dpx u| \norm v^2_h \right ] \\
& \le C \norm v^2_h \left [ \epsilon \norm{ \Dpx^3 \Dmx^2  u}_h + c(\epsilon) k_0 \right ],
\end{aligned}
\end{equation*}
where $C$ is a constant independent of $\Dx$. Hence we can conclude that, using \eqref{eq:fifthorder} 
\begin{equation}
\label{eq:discrete2}
\begin{aligned}
\frac{d}{dt} \norm v^2_h & \le \nu_3  \norm v^3_h + \nu_4  \norm v^2_h,
\end{aligned}
\end{equation}
where $ \nu_3 $ and $\nu_4$ are constants independent of $\Dx$.
So, from \eqref{eq:discrete2}, it is clear that 
\begin{align*}
0\le \norm v^2_h \le y,
\end{align*} 
where $y$ is the solution of the initial-value problem
\begin{equation}
\label{eq:initialvalue}
\begin{aligned} 
\frac{dy}{dt} =  \nu_3 y^{\frac{3}{2}} + \nu_4 y, \quad y(0)= y_0 \ge \norm{v(\cdot, 0)}_h^2.
\end{aligned}
\end{equation}
It can be shown that the solution of \eqref{eq:initialvalue} is finite for 
\begin{align*}
t \le t_{\infty} = \frac{2}{\nu_4} \log ( 1 + \frac{\nu_4}{\nu_3 y_0}).
\end{align*}
There is no possibility of choosing constants scuh that $t_{\infty} = \infty$. However, we can always find a constant $C$ such that \eqref{eq:exis5} holds for $0\le t \le T_1 = t_{\infty}/2$. The estimate \eqref{eq:exis4} then follows from \eqref{eq:fifthorder}. Also the estimates \eqref{eq:exis2} and \eqref{eq:exis3} follows from Lemma~\ref{lemma:lemma-3} and \eqref{eq:exis4}. This completes the proof of the theorem. 
\end{proof}

Now we can use the Theorem ~\ref{theo:localexistence} to conclude that the Kawahara equation \eqref{eq:kawahara} has a solution $u(x,t)$ in $0 \le t\le T_1$. To do that, let us first assume that $u_{\Dx}$ is a grid function, i.e., a function $u_{\Dx}: [0,T]\rightarrow \R^{N}$, where $u_{\Dx}(x_i,t)= u_i(t)$, such that $u_0 = u_N$. We will also denote $\psi_{\Dx}(x,t)$ to be the Fourier series of $u_{\Dx}(x,t)$. Now we are ready to use the Lemma ~\ref{lemma:lemma-2} and ~\ref{theo:localexistence} to conclude immediately that
\begin{equation}
\label{eq:eq1}
\begin{aligned}
\norm {\frac{ \partial \psi_{\Dx}(.,t)}{\partial t}}^2 \le c_1, \quad \text{$0 \le t \le T_1$},
\end{aligned}
\end{equation}
\begin{equation}
\label{eq:eq2}
\begin{aligned}
\norm {\frac{\partial^3 \psi_{\Dx}(.,t)}{\partial x^3}}^2 \le c_2, \quad \text{$0 \le t \le T_1$},
\end{aligned}
\end{equation}
\begin{equation}
\label{eq:eq3}
\begin{aligned}
\norm {\frac{\partial^5 \psi_{\Dx}(.,t)}{\partial x^5}}^2 \le c_3, \quad \text{$0 \le t \le T_1$},
\end{aligned}
\end{equation}
where $c_1$, $c_2$ and $c_3$ are constants depending on $f(x)$ but not on $\Dx$. Now the above inequalities \eqref{eq:eq1} and \eqref{eq:eq2} imply that the sequence ${\lbrace \psi_{\Dx} \rbrace}_{\Dx>0}$ is bounded and equicontinuous in both $x$ and $t$. The Arzela-Ascoli theorem guarantees the existence of a subsequence of ${\lbrace \psi_{\Dx} \rbrace}_{\Dx>0}$ which converges to some function $u(x,t)$ as $\Dx \rightarrow 0$. The rest of the argument is standard. Both $\frac{\partial^3 \psi_{\Dx}}{\partial x^3}$ and $\frac{\partial^5 \psi_{\Dx}}{\partial x^5}$ converges to $\frac{\partial^3 u}{\partial x^3}$ and $\frac{\partial^3 u}{\partial x^3}$ respectively in the $L^2$ sense. It is also easy to see that $u(x,t)$ is a solution of the Kawahara equation because of the very definition of $\psi_{\Dx}(x,t)$.

\subsubsection{ Proof of Global existence}
In this section we are going to prove the existence of the global solution of the equation \eqref{eq:kawahara}. At first we are going to prove the following Lemmas, which we will use later.
\begin{lemma}
\label{lemma-4}
Let $ u(x,t)$ be a solution of the problem \eqref{eq:kawahara}. Then there exist a constants $\alpha_1, \alpha_2$ such that
\begin{equation}
\label{eq:cont1}
\begin{aligned}
\int_{0}^{1} u^2(x,t)\,dx = \int_{0}^{1} u^2(x,0)\,dx = \int_{0}^{1} f^2\,dx = \alpha_1 
\end{aligned}
\end{equation}
\begin{equation}
\label{eq:cont2}
\begin{aligned}
\int_{0}^{1} \left( \frac{1}{3} u^3 -  u_x^2 - u_{xx}^2 \right)\,dx  = \int_{0}^{1} \left(\frac{1}{3} f^3- {f'}^2 -  {f''}^2 \right)\,dx = \alpha_2
\end{aligned}
\end{equation}
\end{lemma}

\begin{proof}
To prove \eqref{eq:cont1} we start by multyplying the equation \eqref{eq:kawahara} by $u $ and integrate by parts in space, yields
\begin{align*}
\int_{0}^{1} uu_t \,dx & = \int_{0}^{1} -u^2 u_x - uu_{xxx} +  uu_{xxxxx} \,dx \\
& = -\int_{0}^{1} (\frac{1}{3} u^3)_x \,dx - \int_{0}^{1} (uu_{xx} - \frac{1}{2}u_x^2)_x \,dx - \int_{0}^{1} u_x u_{xxxx}\,dx \\
&= -\int_{0}^{1} (\frac{1}{3} u^3)_x \,dx - \int_{0}^{1} (uu_{xx} - \frac{1}{2}u_x^2)_x \,dx \\
& \qquad - \int_{0}^{1} (u_xu_{xxx} - \frac{1}{2}u_{xx}^2)_x \,dx = 0.
\end{align*}
Since all the boundary terms vanishes due to the periodic boundary condition. Hence this established \eqref{eq:cont1}. 

To prove \eqref{eq:cont2}, we start by multyplying \eqref{eq:kawahara} by $u^2 $ and integrate by parts in space, yields
\begin{align*}
\int_{0}^{1} u^2u_t \,dx & = \int_{0}^{1} - u^3 u_x  - u^2u_{xxx} +  u^2u_{xxxxx} \,dx \\
& = -\int_{0}^{1} (\frac{1}{4} u^4)_x \,dx + 2 \int_{0}^{1} (uu_x)u_{xx} \,dx -  2 \int_{0}^{1} (uu_x) u_{xxxx}\,dx\\
& =2 \int_{0}^{1} [ -u_t -  u_{xxx} + u_{xxxxx} ]u_{xx} \,dx \\
& \qquad \qquad - 2 \int_{0}^{1}[ - u_t - u_{xxx} + u_{xxxxx} ] u_{xxxx}\,dx \\
& =   2 \int_{0}^{1} u_{tx}u_x \,dx - 2 \int_{0}^{1} u_{xx}u_{xxx} \,dx + 2 \int_{0}^{1} u_{xx}u_{xxxxx} \,dx \\
 - 2 & \int_{0}^{1} u_{tx}u_{xxx} \,dx + 2  \int_{0}^{1} u_{xxx}u_{xxxx} \,dx - 2 \int_{0}^{1} u_{xxxx}u_{xxxxx} \,dx \\
&= 2 \int_{0}^{1} u_{tx}u_x \,dx + 2 \int_{0}^{1} u_{txx}u_{xx} \,dx.
\end{align*}
From this we can conclude that
\begin{align*}
\frac{d}{dt}\int_{0}^{1} [\frac{1}{3}u^3 -  u_x^2 - u_{xx}^2 ] \,dx = 0,
\end{align*}
and consequently \eqref{eq:cont2} follows from the above equation.
\end{proof}
\begin{lemma}
\label{lemma-5}
Let $ u(x,t)$ be a solution of the problem \eqref{eq:kawahara}. Then there exists a constant $\alpha$ such that
\begin{equation}
\label{eq:global1}
\begin{aligned}
\max | u_{x} (x, t)| \le \alpha
\end{aligned}
\end{equation}
\begin{equation}
\label{eq:global2}
\begin{aligned}
\norm {v}^2  \le e^{\gamma t} \norm { - f f^{'} -  f^{'''} + f^{'''''}}^2, \qquad \text {$ v = \frac{\partial u}{\partial t}$}. 
\end{aligned}
\end{equation} 
\end{lemma}
\begin{proof} 
From \eqref{eq:cont1}, it follows that
\begin{align*}
 \norm {u_{xx}}^2 & \le \frac{1}{3} \int_{0}^{1} |u^3| \,dx  + \norm { u_x}^2 + |\alpha_2|\\
& \le \frac{1}{3} \left ( c(\epsilon) \norm u + \epsilon \norm {u_{xx}} \right ) \norm {u}^2 +  \left ( c(\epsilon) \norm {u}^2 + \epsilon \norm {u_{xx}}^2 \right )+ |\alpha_2|\\
& = \frac{1}{3} \left ( c(\epsilon) \sqrt {\alpha_1} + \epsilon \norm {u_{xx}} \right ) \alpha_1+ \left ( c(\epsilon) {\alpha_1} + \epsilon \norm {u_{xx}}^2 \right )+ |\alpha_2|.
\end{align*}
Now we can rewrite the above inequality in the following form 
\begin{equation}
\label{eq:bound}
\begin{aligned} 
a \norm{u_{xx}}^2 - b \norm{u_{xx}} - c & \le 0,
\end{aligned}
\end{equation}
for some constants $a,b,c$, where $ a = 1-\epsilon$, $b = \frac{1}{3} \alpha_1 \epsilon$ and $c = \frac{1}{3} c(\epsilon) \sqrt {\alpha_1} + c(\epsilon) \alpha_1 + |\alpha_2|$. Now it is easy to see that \eqref{eq:bound} gives,
$$
\left(  \sqrt a \norm {u_{xx}} - \frac{b}{2 \sqrt a} \right)^2 \le c + \frac{b^2}{4a}.
$$
From the above relation, it is clear that $\norm {u_{xx}} \le \alpha_3 $, for some constant $\alpha_3$. Again using the interpolation inequality, $\norm {u_x} \le \left ( c(\epsilon) \norm u + \epsilon \norm {u_{xx}} \right )$, we can conclude that $\norm {u_x} \le \alpha_4$, for some constant $\alpha_4$. Also we can use a similar type interpolation inequality to conclude that \eqref{eq:global1} holds.

Now the function $ v(x,t)= \frac{\partial u}{\partial t} $ satisfies,
\begin{equation*}
\begin{aligned}
\frac{dv}{dt}& =- vu_x - uv_x - v_{xxx} +  v_{xxxxx}
\end{aligned}
\end{equation*}
Multiplying the above equation by $v$ and integrating in space yields,
\begin{equation*}
\begin{aligned}
\frac{1}{2} \frac{d}{dt} \norm v^2 & = -(v, vu_x) - (v, uv_x) - (v, v_{xxx}) + (v, v_{xxxxx})\\
& \le (v^2 , u_x)
\end{aligned}
\end{equation*} 
Now \eqref{eq:global1} gives, 
\begin{align*}
\frac{d}{dt} \norm v^2 \le  \max |u_x| \norm v^2 \le C \norm v^2,
\end{align*} 
which implies
\begin{align*}
\norm v^2 & = \norm {\frac{\partial u(\cdot, t)}{\partial t}}^2 \le e^{\gamma t} \norm {\frac{\partial u(\cdot, 0)}{\partial t}}^2\\
& = e^{\gamma t} \norm { -ff' -  f^{'''} +  f^{'''''}}^2
\end{align*}
i.e. we established \eqref{eq:global2}.
\end{proof}

\noindent Now to get a bound on $\norm {u_{xxx}}$, we proceed as follows:
\begin{align*}
\frac{1}{2}\frac{d}{dt}\int_{0}^{1} u_{xxx}^2 \,dx & =  \int_{0}^{1} u_{xxx} u_{xxxt}\,dx\\
& = - \int_{0}^{1} u_{xxxx} u_{xxt}\,dx =  \int_{0}^{1}  u_{xxxxx} u_{xt}\,dx \\
& =  \int_{0}^{1} u_{xt} ( u_t + uu_x + u_{xxx})\,dx\\
& = \int_{0}^{1} v v_x \,dx +  \int_{0}^{1} u u_x u_{xt} \,dx - \frac{1}{2} \frac{d}{dt}  \int_{0}^{1} u_{xx}^2 \,dx\\
& = \int_{0}^{1} (\frac{1}{2}v^2)_x \,dx +  \int_{0}^{1} u u_x u_{xt} \,dx - \frac{1}{2} \frac{d}{dt}  \int_{0}^{1} u_{xx}^2 \,dx\\
& = \int_{0}^{1} v u_{x}^2 \,dx + \int_{0}^{1} v u u_{xx} \,dx + \beta \frac{1}{2} \frac{d}{dt}  \int_{0}^{1} u_{xx}^2 \,dx\\
& \le (\max |u_x|)( \norm{u_x} \norm {v}) + (\max |u|)( \norm{u_{xx}} \norm {v}) \\
& \qquad \qquad + \frac{1}{2} \frac{d}{dt}  \int_{0}^{1}u_{xx}^2 \,dx
\end{align*}
Since the use of the interpolation inequality allows us to conclude that all the terms $\max|u|, \max|u_x|, \norm {u_x}, \norm{u_{xx}}$ are bounded by some constant, hence from the above relation we end up with
\begin{align*}
\frac{d}{dt} \int_{0}^{1} \left( \frac{1}{2} u_{xxx}^2 - \frac{1}{2} u_{xx}^2\right)\,dx \le \alpha \norm{v}. 
\end{align*}
Therefore we conclude that,
\begin{equation}
\label{eq:third}
\begin{aligned}
\norm {u_{xxx}}^2 \le C_1 e^{\gamma t} ( \norm { - ff' -  f^{'''} +  f^{'''''}}).
\end{aligned}
\end{equation}
Again we can now use the Kawahara equation \eqref{eq:kawahara} and triangular inequality to conclude that 
From the above relation we conclude that,
\begin{equation}
\label{eq:fifth}
\begin{aligned}
\norm {u_{xxxxx}} \le C_1 e^{\gamma t} ( \norm { - ff' -  f^{'''} +  f^{'''''}}) + C_2,
\end{aligned}
\end{equation}
where $C_1$ and $C_2$ are constants. One can see that the bound \eqref{eq:global2} guarantees that $\frac {\partial u}{\partial t}$ is square integrable for every $t$ and 
\eqref{eq:third}, \eqref{eq:fifth} that the problem \eqref{eq:kawahara} with the initial function $u(x,T_1)$ instead of $f(x)$ has a solution for $T_1 \le  t \le T_2 = 2T_1$. Consequently, we get a solution of the Kawahara equation for $0 \le t \le T$. Now to obtain the existence of sloutions for all $t>0$, we will repeat the extension procedure. For this purpose suppose that existence can be proven only in $0 <T < \infty$. Now if we look at the expression for $t_{\infty}$, we find that only $y_0$ depends on $t$. But $y_0$ can, because of \eqref{eq:global2}, be chosen to hold in the whole interval $0\le t \le T$. Consequently, if we consider problem \eqref{eq:kawahara} with $f(x)= u(x,\tau)$ for some $\tau$ sufficiently close to $T$, we can, by using the local procedure, get existence for values of $t$ lying outside $0 \le t \le T$.

\begin{remark}
To prove the existence in the lower half-plane, we proceed as follows:
Consider the following equation
\begin{equation}
\label{eq:lowerhalf}
\begin{aligned}
- u_t & = u u_x + u_{xxx} - u_{xxxxx},
\end{aligned}
\end{equation}
which arises when making the transformations $ x \rightarrow (1-x)$ and $ t \rightarrow -t$. Now if we substitute \eqref{eq:lowerhalf} for \eqref{eq:kawahara}, then we can prove existence of solution in the lower half plane. This is possible since we get an interval $T_1 \le t \le 0$ in Theorem ~\ref{theo:localexistence}. Consequently, this concludes the proof of the existence part of Theorem ~\ref{theo:maintheorem}. 
\end{remark}

\section{ Fully Discrete semi-Implicit Scheme:}
We know that in order to have a completely satisfactory numerical method, one must seek a fully discrete scheme that reduces the actual computation to a solution of a finite set of algebric equations. So, here we will consider a semi-implicit fully discrete scheme and show the convergence of the solution of that fully discrete scheme to the solution of \eqref{eq:kawahara}. To do that, let us select a time step $\Dt>0$, and write $t^n = n\Dt$. The value of our differnce approximation at $(j\Dx, n\Dt)$ will be denoted by $u_j^n$ in the fully discrete case. To simplify the notation, we introduce the finite difference operator:
\begin{align*}
\Dpt u_j^n = \frac{1}{\Dt}(u_j^{n+1} - u_{j}^n).
\end{align*}
We propose the following semi-implicit fully-discrete aproximation to \eqref{eq:kawahara}, given by
\begin{equation}
\begin{aligned}
\label{eq:fullydiscrete1}
u_j^{n+1} & =  \bar{u}_j^{n} - \frac{\Dt}{3}[ \bar{u}_j^{n} \Dcx u_j^{n}  + \Dcx (u_j^{n})^2] - \Dt \Dmx \Dpx^{2} u_j^{n+1} +\Dt \Dpx^3 \Dmx^2 u_j^{n+1},\\
&\qquad \qquad \text{where} \quad \bar{u}_j = \frac{1}{2}( u_{j+1} + u_{j-1}).
\end{aligned}
\end{equation}
Now keeping in mind that $ \Dcx u_j^2 = 2 \bar{u}_j \Dcx u_j$, we can rewrite the above scheme as
\begin{equation}
\begin{aligned}
\label{eq:fullydiscrete}
u_j^{n+1} & =  \bar{u}_j^n - \Dt \bar{u}_j^{n} \Dcx u_j^{n} - \Dt \Dmx \Dpx^{2} u_j^{n+1} + \Dt \Dpx^3 \Dmx^2 u_j^{n+1}.
\end{aligned}
\end{equation}
with the initial condition 
\begin{equation}
\begin{aligned}
\label{eq:fullydiscrete_initial}
u_i^{0}= f(x_i), \quad \text{$ i= 1,2,...,N$},
\end{aligned}
\end{equation}
and the boundary condition
\begin{equation}
\begin{aligned}
\label{eq:fullydiscrete_boundary}
u_i^{n}= u_{i+N}^{n} \quad \text{ for all $i$ and $n$}
\end{aligned}.
\end{equation}


\subsection{ Local existence of solution:}

To show the local existence of solution to \eqref{eq:kawahara}, we will use similar arguments to the ones used in the semi discrete case. First we will state the main theorem, similar to Theorem ~\ref{theo:localexistence}.
\begin{theorem}
\label{theo:fulllocalexistence}
There exists a time $T>0$ and constants $k_i$, $i= 0,1,2,3,4$ independent of $\Dx$ but dependent on $f(x)$ and its derivatives of order five and lower, such that
\begin{equation}
\label{eq:fullexis1}
\begin{aligned}
\norm {u^n}_{h} \le k_0,\qquad \qquad \text{ $0\le n\Dt \le T$} 
\end{aligned}
\end{equation}
\begin{equation}
\label{eq:fullexis2}
\begin{aligned}
| u^n(x_i)| \le k_1, \qquad \qquad \text{ $0\le n\Dt \le T$}
\end{aligned}
\end{equation} 
\begin{equation}
\label{eq:fullexis3}
\begin{aligned}
\norm {\Dmx \Dpx^{2} u^n}_{h} \le k_2,  \quad \text{ $0\le n\Dt \le T$}
\end{aligned}
\end{equation} 
\begin{equation}
\label{eq:fullexis4}
\begin{aligned}
\norm {\Dpx^{3} \Dmx^{2} u^n}_{h}  \le k_3, \quad \text{ $0\le n\Dt \le T$}
\end{aligned}
\end{equation} 
and finally with,
\begin{equation*}
\begin{aligned}
v^n = \Dpt u^{n-1}, \quad \text{$ n \in \N_0$} 
\end{aligned}
\end{equation*}
\begin{equation}
\label{eq:fullexis5}
\begin{aligned}
\norm {v^n}_{h}  \le k_4, \quad \text{ $0\le n\Dt \le T$}.
\end{aligned}
\end{equation} 
\end{theorem}

\noindent Before giving a proof of the above theorem, we will first prove the following usefull Lemma:
\begin{lemma}
\label{lemma:fullydiscrete}
Let $ u_j^n$ be a solution of the difference scheme \eqref{eq:fullydiscrete} and define $v^n = \Dpt u^{n-1}$. Then the following two estimates hold
\begin{equation}
\label{eq:fullest1}
\begin{aligned}
\norm {u^{n+1}}^2 + \Dt \Dx \left ( \norm { \Dmx \Dpx^2 u^{n+1}}^2  + \norm { \Dpx {\Dmx} u^{n+1}}^2 + \frac{1}{8 \lambda} \norm {\Dcx u^n}^2 \right) & \le \norm {u^{n}}^2 ,
\end{aligned}
\end{equation}
also we have, 
\begin{equation}
\label{eq:fullest2}
\begin{aligned}
& \norm {v^{n+1}}^2 +  \Dt \Dx \left(  \norm{ \Dpx \Dmx {v^{n+1}}}^2  + \norm{ \Dmx \Dpx^2 {v^{n+1}}}^2 + \frac{1}{8 \lambda} \norm {Dv^n}^2 \right) \\
& \qquad \qquad \qquad \le \norm {v^n}^2  + C\Dt \max |Du^n| \norm{v^n}^2,
\end{aligned}
\end{equation}
provided the CFL condition
\begin{equation}
\label{eq:cfl}
\begin{aligned}
\qquad \qquad \frac{\Dt}{\Dx^{3/2}} \norm {u^n} ( 1 + 24 \frac{\Dt}{\Dx^{3/2}} \norm {u^n}) & \le \frac{3}{2},
\end{aligned}
\end{equation}
\end{lemma}

\begin{proof}
First of all, just to avoid clumsy notations, we will drop the index $j$ from our notation, and use the notation $u, Du$ for $u_j , \Dcx u_j$ respectively where $j$ is fixed. We first study the ``Burgers'' term $\Dt \bar{u} Du$. Let $u$ be a gridfunction and set
\begin{equation}
\label{eq:dis1}
\begin{aligned}
w = \overline{u} - \Dt \bar{u} Du.
\end{aligned}
\end{equation}
Set $\lambda = \Dt/\Dx$, and we will use the following CFL condition:
\begin{equation}
\label{eq:cflsemi}
\begin{aligned}
\frac{\Dt}{\Dx}\max |u| ( 1 + 24 \frac{\Dt}{\Dx} \max |u|) \le \frac {\Dt}{\Dx^{3/2}} \norm {u} ( 1 + 24 \frac{\Dt}{\Dx^{3/2}} \norm {u}) \le \frac{3}{2}
\end{aligned}
\end{equation}
Multiplying \eqref{eq:dis1} by $\bar{u}$, we have
\begin{align*}
\frac{1}{2} w^2 &= \frac{1}{2} \bar{u}^2 - \Dt\frac{1}{2} \bar{u} D{u}^2 + \frac{1}{2} ( w -\bar{u})^2 \\
& = \frac{1}{2} \bar{u^2} - \Dt\frac{1}{2} \bar{u} D{u}^2 + \Dt^2\bar{u}^2 (Du)^2 +  \frac{1}{2} ( \bar{u}^2 -\bar{u^2}). 
\end{align*}
We will use the following relations
\begin{align*}
\frac{1}{4} (a+b)(a^2 -b^2) = \frac{1}{3}(a^3 -b^3) -\frac{1}{12} (a-b)^3,
\end{align*}
and
\begin{align*}
\frac{1}{4}(a+b)^2 -\frac{1}{2}(a^2 +b^2) = -\frac{1}{4}(a-b)^2. 
\end{align*}
For a gridfunction, these implies
\begin{align*}
\bar{u} D{u}^2 & = \frac{1}{3} D{u}^3 - \frac{\Dx^2}{12} (Du)^3, \\
\bar{u}^2 -\bar{u^2} & = -\frac{\Dx^2}{4} (Du)^2.
\end{align*}
Therefore,\\
\begin{equation}
\label{eq:discrete1}
\begin{aligned}
\frac{1}{2} w^2 & = \frac{1}{2} \bar{u^2} - \frac{\Dt}{6} D{u}^3 + \Dt^2\bar{u}^2 (Du)^2 +  \frac{\Dt \Dx^2}{24}
(Du)^3 - \frac{\Dx^2}{8} (Du)^2\\
& = \frac{1}{2} \bar{u^2} - \frac{\Dt}{6} D{u}^3 - \frac{\Dx^2}{16} (Du)^2\\
& \qquad \qquad +  \frac{\Dx^2}{24}(Du)^2 \left ( \lambda \Dx Du + 24 \lambda^2 \bar{u}^2 - \frac{3}{2} \right )
\end{aligned}
\end{equation}

Now summing \eqref{eq:discrete1} over $j$ and using the CFL condition \eqref{eq:cflsemi} we conclude,
\begin{equation}
\label{eq:dis2}
\begin{aligned}
\norm w^2 + \frac{\Dx^2}{8} \norm {Du}^2 \le \norm u^2.
\end{aligned}
\end{equation}
 
Now we are in a postion to study the full difference scheme by adding the ``Airy term'' $\Dt \Dmx \Dpx^{2} u_j^{n+1}$, and the ``Kawahara term'' $\Dt \Dpx^3 \Dmx^2 u_j^{n+1} $. 
Thus the full difference scheme \eqref{eq:fullydiscrete} can be written as
\begin{equation}
\label{eq:dis3}
\begin{aligned}
v = w - \Dt \Dmx \Dpx^{2} v + \Dt \Dpx^3 {\Dmx}^2 v .
\end{aligned}
\end{equation}
Multiplying the above equation \eqref{eq:dis3} by $v$, and summing over $j$, we have
\begin{align*}
\frac{1}{2} \norm {v}^2 + \sum_{j} \frac{\Dx}{2} (v_j-w_j)^2 & = \frac{1}{2} \norm {w}^2 - \Dt(v , \Dmx \Dpx^{2} v )+ \Dt (v, \Dpx^3 \Dmx^2 v )\\
& =  \frac{1}{2} \norm{w}^2  - \frac{ \Dx \Dt}{2} \norm{ \Dpx \Dmx v}^2 - \frac{ \Dx \Dt}{2} \norm{ \Dmx \Dpx^2 v}^2
\end{align*}
Hence finally we have,
\begin{equation}
\label{eq:dis4}
\begin{aligned}
\norm {v}^2 + \Dx \Dt \norm{ \Dpx \Dmx v}^2 + \Dx \Dt \norm{ \Dmx \Dpx^2 v}^2 \le \norm {w}^2 
\end{aligned}
\end{equation}

Now combining two results coming from \eqref{eq:dis2} and \eqref{eq:dis4}, we conclude that
\begin{equation}
\label{eq:dis5}
\begin{aligned}
\norm {u^{n+1}}^2 + \Dx \Dt \norm{ \Dpx \Dmx {u^{n+1}}}^2 + \Dx \Dt \norm{ \Dpx \Dmx^2 {u^{n+1}}}^2 + \frac{\Dx^2}{8} \norm {Du^n}^2  \le \norm {u^n}^2, 
\end{aligned}
\end{equation}
which is exactly \eqref{eq:fullest1}. 
To prove \eqref{eq:fullest2}, we will proceed as follows: \\
First we define
\begin{align*}
v^n = \Dpt u^{n-1}, \quad \text{$ n \in \N_0$}
\end{align*}

Now the gridfunction satisfies 
\begin{equation}
\label{eq:dis6}
\begin{aligned}
v^{n+1} = \bar{v}^n - \Dt (\bar{v}^n Du^n + \bar{u}^n Dv^n ) +\Dt^2 \bar{v}^n Dv^n - \Dt \Dmx \Dpx^{2} v^{n+1} + \Dt \Dpx^3 \Dmx^2 v^{n+1}.
\end{aligned}
\end{equation}
Set
\begin{align*}
w = \bar{v} - \Dt D(uv) + \frac{\Dt^2}{2} Dv^2
\end{align*}
We proceed as before and multiply this by $\overline{v}$ to find
\begin{align*}
\frac{1}{2} w^2 & = \frac{1}{2} \bar{v^2} + \frac{\Dt^2}{2}\left( D(uv) - \frac{\Dt}{2} Dv^2 \right)^2 \\
& \qquad \qquad - \Dt (\bar{u}\bar{v} Dv + \bar{v}^2 Du ) + \Dt^2 \bar{v}^2 Dv - \frac{\Dx^2}{8} (Dv)^2
\end{align*}
Now we will use the following relations
\begin{align*}
\frac{1}{2}\left( D(uv) - \frac{\Dt}{2} Dv^2 \right)^2 & \le (D(uv))^2 + \Dt^2 \bar{v}^2 (Dv)^2\\
& \le 2 \bar{u}^2 (Dv)^2 + 2 \bar{v}^2 (Du)^2 + \Dt^2 \bar{v}^2 (Dv)^2,\\
\bar{v}^2 Dv & = \frac{1}{3} Dv^3 - \frac{\Dx^2}{12} (Dv)^3,\\
\bar{u}\bar{v} Dv + \bar{v}^2 Du & = \frac{1}{2}D(uv^2) + \frac{1}{2} \bar{v}^2 Du - \frac{\Dx^2}{8}(Dv)^2 Du.
\end{align*}
Using this
\begin{equation}
\label{eq:fuldiscrete2}
\begin{aligned}
\frac{1}{2} w^2 & = \frac{1}{2} \bar{v^2} - \Dt D \left (\frac{1}{2} u v^2 -\frac{\Dt}{3} v^3 \right) - \frac{\Dt}{2} \bar{v}^2 Du - \frac{\Dt \Dx^2}{8} (Dv)^2 Du \\
& \qquad + \Dt^2 \left( 2 \bar{u}^2 (Dv)^2 + 2 \bar{v}^2 (Du)^2 + \Dt^2 \bar{v}^2 (Dv)^2 - \frac{\Dx^2}{12}(Dv)^3\right)\\
& \qquad \qquad - \frac{\Dx^2}{8} (Dv)^2.
\end{aligned}
\end{equation}
Now our aim should be to balance all the positive terms with $ \Dx^2 (Dv)^2$. For that we will proceed as follows:
\begin{align*}
\Dx^2 (Dv)^2 &\le 4 \bar{v^2},\\
\Dt^2 \bar{v}^2 (Du)^2 &\le C \Dt \bar{v^2}|Du|,\\
\Dt^2 \bar{v}^2 & \le 2 \left( \max_{j} |u^n|^2 + \max_{j} |u^{n-1}|^2 \right),\\
\Dx^2 |Dv| & \le \frac{1}{\lambda}|\bar{u}|,  
\end{align*}
where the constant $C$ depends on $\norm u$. Now using the above inequalities in \eqref{eq:fuldiscrete2}, we get
\begin{align*}
\frac{1}{2} w^2 & = \frac{1}{2} \bar{v^2} - \Dt D \left (\frac{1}{2} u v^2 -\frac{\Dt}{3} v^3 \right) + C\Dt \bar{v^2} |Du|\\
& \qquad + \lambda^2 \Dx^2 (Dv)^2 \left( 4 \max |u^n|^2 + 2 \max |u^{n-1}|^2 + \frac{1}{12 \lambda} \max |u^n|\right) \\
& \qquad \qquad  - \frac{\Dx^2}{8} (Dv)^2.
\end{align*}
Now we choose $\Dt$ in such a way that,
\begin{equation}
\label{eq:fullycfl}
\begin{aligned}
\lambda^2 \left( 4 \max |u^n|^2 + 2 \max |u^{n-1}|^2 \right) + \frac{ \lambda}{12} \max |u^n| \le \frac{1}{16}.
\end{aligned}
\end{equation}
Now after summing over $j$, we have
\begin{equation}
\label{eq:dis8}
\begin{aligned}
\norm w^2 + \frac{\Dx^2}{8} \norm {Dv}^2 \le \norm v^2 + C\Dt \text{max}|Du| \norm{v}^2.
\end{aligned}
\end{equation}

\noindent Now using that
\begin{align*}
v^{n+1} = w - \Dt \Dpx \Dmx^{2} v^{n+1} + \Dt \Dpx^3 \Dmx^2 v^{n+1},
\end{align*}
we get 
\begin{equation}
\label{eq:dis9}
\begin{aligned}
\norm {v^{n+1}}^2 +  \Dx \Dt \norm{ \Dpx \Dmx {v^{n+1}}}^2 & + \Dx \Dt \norm{ \Dmx \Dpx^2 {v^{n+1}}}^2 + \frac{\Dx^2}{8} \norm {Dv^n}^2  \\
& \le \norm {v^n}^2  + C\Dt \max |Du^n| \norm{v^n}^2.
\end{aligned}
\end{equation}
Which proves the second part \eqref{eq:fullest2} of the Lemma.
\end{proof}

\noindent Now we are ready to prove Theorem ~\ref{theo:fulllocalexistence}.
\begin{proof}(of Theorem ~\ref{theo:fulllocalexistence})\\ 

\noindent First note that, from the definition of $v^n$, \eqref{eq:fullydiscrete} can be rewritten as,
\begin{equation}
\label{eq:dis10}
\begin{aligned}
 v^{n+1} = \Dpt u^n = \frac{\Dx}{2\lambda} \Dpx \Dmx u^n - \bar{u}^n Du^n - \Dmx \Dpx^2 u^{n+1} + \Dpx^3 \Dmx^2 u^{n+1}
\end{aligned}
\end{equation}
Therefore,
\begin{align*}
\norm {\Dpx^3 \Dmx^2 u^{n+1}} & \le \norm {v^{n+1}} + \norm {\bar{u}^n Du^n} + \frac{\Dx}{2\lambda} \norm {\Dpx \Dmx u^n} + \norm {\Dmx \Dpx^2 u^{n+1}}\\
& \le \norm {v^{n+1}} + \max| Du^n| \norm{u^n} + \norm {\Dmx \Dpx^2 u^{n+1}} + C\\
& \le \norm {v^{n+1}} + C_1 \norm{\Dpx^3 \Dmx^2 u^{n+1}} + C_2 \norm{{\Dmx} \Dpx^2 u^{n}}\\
& \le \norm {v^{n+1}} + C_1 \norm{\Dpx^3 \Dmx^2 u^{n+1}} \\
& \qquad \qquad +\left( \Dt C_2 \norm{{\Dmx}\Dpx^2 v^{n+1}} + C_3 \norm{{\Dmx} \Dpx^2 u^{n+1}} \right )\\
& \le \norm {v^{n+1}} + C_1 \norm{\Dpx^3 \Dmx^2 u^{n+1}} \\
& \qquad \qquad \qquad \qquad + \left( \norm{v^n}^2 ( 1 + C \Dt \max|Du^n|) \right)^{\frac{1}{2}},
\end{align*}
where we have used \eqref{eq:dis9} to estimate $\Dt \norm {\Dmx \Dpx^2 v^{n+1}}$. Now by the CFL condition \eqref{eq:cfl}, $\Dt \max |Du| \le C $ for some constant $C$. Hence, finally we have
\begin{equation}
\label{eq:dis11}
\begin{aligned}
\norm {\Dpx^3 \Dmx^2 u^{n+1}} \le c_0 + c_1 \norm {v^{n+1}} + c_2 \norm {v^{n}},
\end{aligned}
\end{equation}
for some constants $c_0, c_1, c_2$ that are independent of $\Dx$

Now using this result in \eqref{eq:dis9}, we have
\begin{align*}
\norm {v^{n+1}}^2 \le \norm {v^n}^2  + \Dt \left( d_1 \norm {v^n}^2 + d_2( \norm {v^n}^3 + \norm {v^n}^2 \norm {v^{n-1}})\right ),
\end{align*}
for constants $d_1$ and $d_2$. Set $ a_n = \norm{v^n}^2$, so that
\begin{align*}
a_{n+1} \le a_n + \Dt \left( d_1 a_n + d_2 \left( a_n^{\frac{3}{2}} + a_n a_{n-1}^{\frac{1}{2}}\right) \right).
\end{align*}

Now let $\alpha$ be the solution of the differential equation 
\begin{align*}
\frac{d\alpha}{dt} = d_1 \alpha + d_2 \alpha^{\frac{3}{2}}, \qquad \alpha_0 = a \ge 0.
\end{align*}

The solution has a blow up time 
\begin{align*}
T^{\infty} = t_0 + \frac{2}{d_1} \log \left( 1 + \frac{d}{d_1} \right).
\end{align*}

Furthermore, for $t < T^{\infty}$, $\alpha$ is a convex function of $t$. We now claim that for $ n\Dt < T^{\infty} (a,0)$, $a_n < \alpha(n\Dt; (0,a))$. This clearly holds for $n =0$ and $n=1$ (since $\alpha$ is increasing in $t$). Assuming that the claim holds for integers up to $n$, we get
\begin{align*}
a_{n+1} & \le \alpha(n\Dt; (0,a))\\
& \quad + \Dt d_1 \alpha(n\Dt; (0,a)) \\
& \quad + \Dt d_2 \left ( \alpha(n\Dt; (0,a))^{\frac{3}{2}} + \alpha(n\Dt; (0,a))\alpha((n-1)\Dt; (0,a))^{\frac{1}{2}}  \right )\\
& \le \alpha(n\Dt; (0,a)) + \Dt \left( d_1 \alpha(n\Dt; (0,a)) + 2 d_2 \alpha(n\Dt; (0,a))^{\frac{3}{2}}\right)\\
& \le \alpha((n+1)\Dt; (0,a)).
\end{align*}
Hence for $ t \le T = T^{\infty}/2$, $\norm{v^n} \le C$ for some constant $C$ independent of $\Dx$. From this all the results in Theorem ~\ref{theo:fulllocalexistence} follows.
\end{proof} 
Therefore, we can follow exactly the approach of the semi discrete case to conclude that $u_{\Dx}$ converges in $L^2$ to some function $u(x,t)$ for $t \le T$. Furthermore, we have that
\begin{align*}
\norm {u_t} \le c_1, \quad \norm {u_{xxx}} \le c_2 \quad  \text {and} \quad \norm {u_{xxxxx} }\le c_3
\end{align*}
Therefore we can show that $u(x,t)$ satisfies \eqref{eq:kawahara}.

\section{uniqueness}

To prove the uniqueness, let us assume that there exists two solutions $u(x,t)$ and $v(x,t)$ of the problem \eqref{eq:kawahara}. Then the function $w = u-v$ satisfies the following equation
\begin{equation}
\label{eq:unique}
\begin{aligned}
w_t & = -(u u_x -v v_x) - w_{xxx} +  w_{xxxxx} \\
& =  -w u_x - v w_x - w_{xxx} +  w_{xxxxx}\\
w(&x,0) = 0.
\end{aligned}
\end{equation} 

Hence, by taking inner product of the above equation with $w$, we have
\begin{align*}
(w,w_t) & = \frac{1}{2} \frac{\partial \norm{w}^2}{\partial t} = - (w, wu_x) - (w, vw_x) -  (w, w_{xxx} ) + (w, w_{xxxxx} )\\
& =  -(w^2, u_x) + (w^2, v_x)/2,
\end{align*}
by periodicity.
Now use of estimate \eqref{eq:global1} implies,
\begin{align*}
\frac{\partial \norm{w}^2}{\partial t} \le C \norm{w}^2,
\end{align*}
for some constant $C$. Now as $\norm{w(\cdot,0)}^2 = 0$, it is clear that $\norm{w(\cdot,t)}^2 = 0$ for all $t$ which consequently implies uniqueness.

\section{Numerical experiments}
\subsection{Numerical experiment  1}
The fully-discrete scheme given by \eqref{eq:fullydiscrete} have been tested on a suitable numerical experiment in order to demonstrate its effectiveness. It is well known that a soliton is a self-reinforcing solitary wave that maintains its shape while it travels at a constant speed. Solitons are caused by a cancellation of nonlinear and dispersive effects in the medium. Several authors [ see \cite{juan}, \cite{mau}, \cite{nils} ] have studied the soliton experiments in the context of both KdV and Kawahara equation. Here we are interested in the soliton experiment for the Kawahara equation only. We are going to compare our result with the results given by \cite{juan}.

Although all the numerical experiments performed in \cite{juan} based on the equation given by
\begin{equation}
\label{eq:juana}
\begin{aligned}
u_t + u_x + uu_x + u_{xxx} & = u_{xxxxx}, 
\end{aligned}
\end{equation}
but it is indeed very easy to see that \eqref{eq:kawahara} and \eqref{eq:juana} are completely equivalent by way of simple change of variables. In \cite{juan}, authors have considered the following scheme for the Kawahara equation \eqref{eq:kawahara}
\begin{align*}
\frac{u_i^{n+1} - u_i^n}{\Dt} + \frac{1}{2} \Dmx [u_i^n]^2 + A u_i^{n+1}=0,
\end{align*} 
where $ A = \Dpx \Dpx \Dmx - \Dpx \Dpx \Dpx \Dmx \Dmx $.
In the case of Kawahara equation given by \eqref{eq:kawahara}, if we consider the initial function in the following form given by
\begin{align*}
u(x,0) = \frac{105}{169} \sech^4 \left (  \frac{1}{2 \sqrt 13} ( x - c)\right ),
\end{align*}
then it is known from \cite{dar} that the explicit solution is given by the following travelling wave 
\begin{align*}
u(x,t)= \frac{105}{169} \sech^4 \left (  \frac{1}{2 \sqrt 13} ( x - \frac{36 t}{169}- c)\right )
\end{align*}
This result can be verified through substitution.\\
Since we know that the behaviour of the exact solution for Kawahara equation, mainly which remains its shape as time grows, it will be interesting to see how the numerical solution given by the scheme \eqref{eq:fullydiscrete} evolves with time.
We will use the following notations: UK scheme - scheme described in this paper and JMO scheme - scheme described as in \cite{juan} and $\norm{u}_{l^2} = \left( \Dx \sum_{k} u_k^2 \right)^{\frac{1}{2}}$.
In order to compare with the existing scheme given by \cite{juan}, we present the $l^2$
errors on a computational domain $[-40,40]$, between exact solution and the solution generated by the $UK$ and $JMO$ schemes in table \ref{tab:1}.

\begin{table}[htbp]
\centering
\begin{tabular}{l|ccc}
Mesh points                 & $UK$               & $JMO$ \\
\hline
4000                        & 2.7e-3             & 1.2e-3  \\
8000                        & 1.4e-3             & 6.2e-4  \\
12000                       & 9.2e-4             & 4.2e-4  \\
16000                       & 7.0e-4             & 3.0e-4 
\end{tabular}
\caption{Numerical Experiment $1$: $l^2$ errors between exact and simulated solutions at time 
  $t=10$ for both $UK$ and $JMO$ schemes . }
\label{tab:1} 
\end{table}

In the following figures we show the behaviour of the numerical solutions at different times. In this case we have used a domain $[-20,50]$, $5000$ mesh points and a CFL number $0.75$. We will compare our results with the results given by \cite{juan}.
\begin{figure}[htbp]
  \centering
    \subfigure {\includegraphics[width=0.45\linewidth]{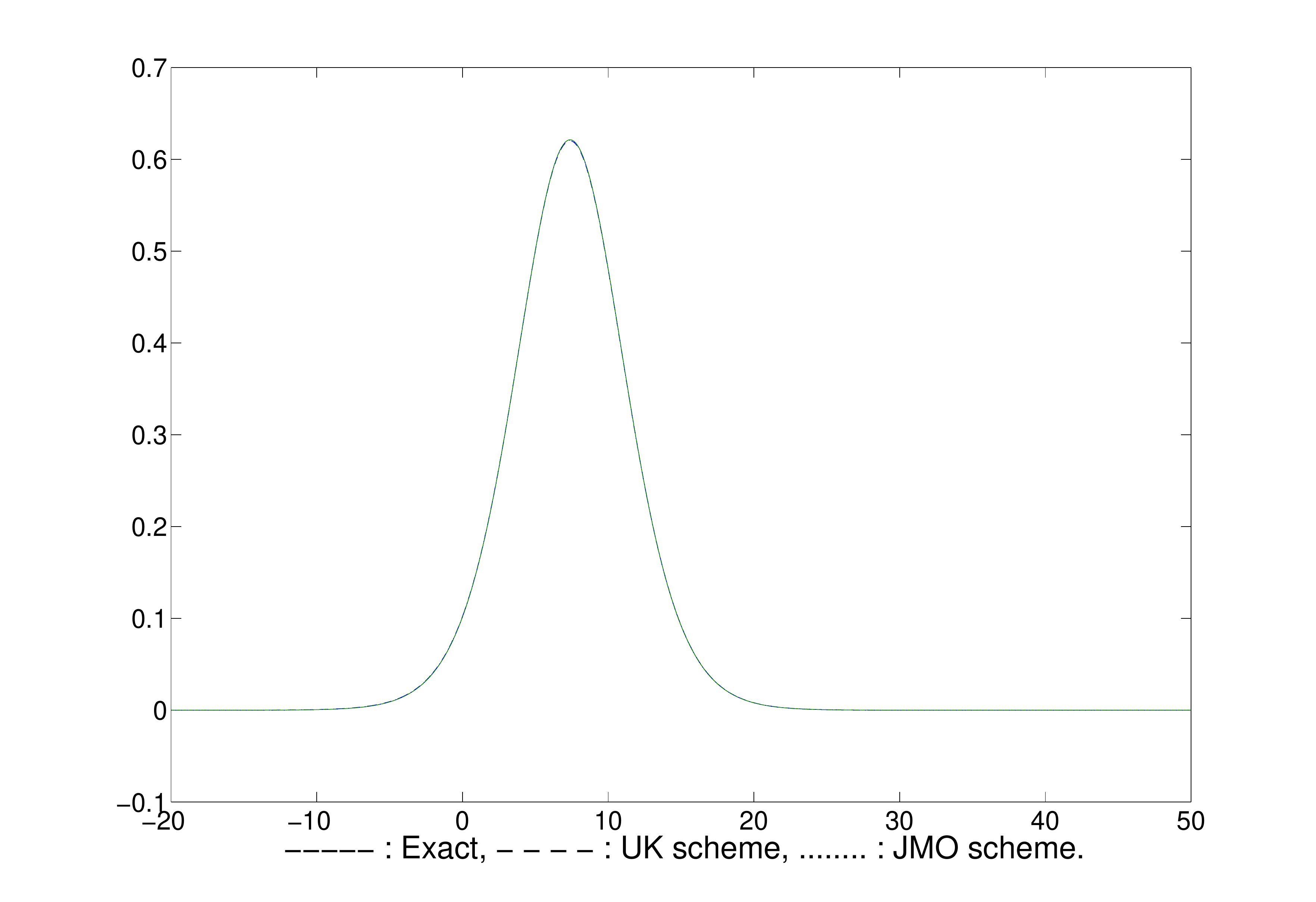}}
    \subfigure {\includegraphics[width=0.45\linewidth]{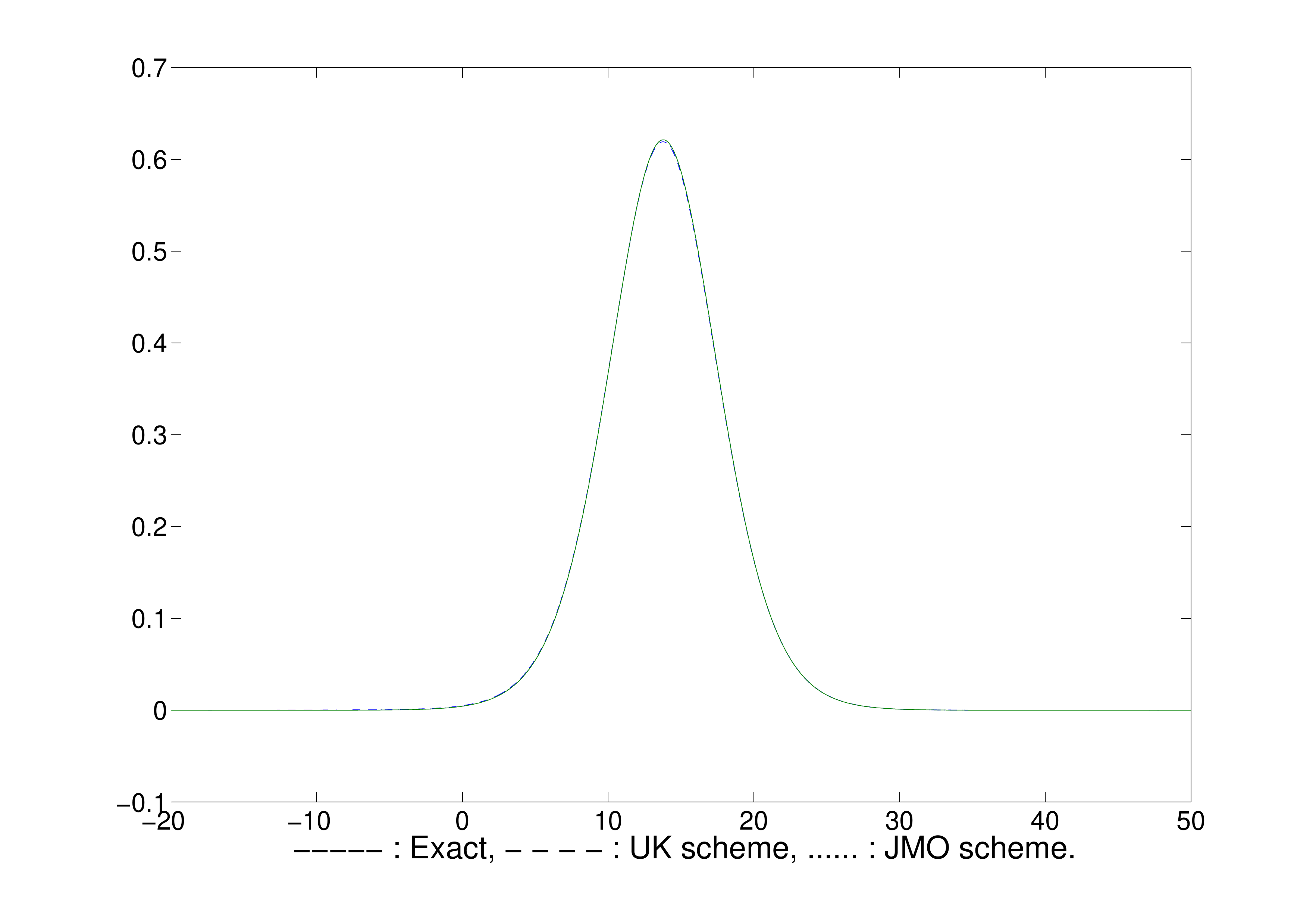}}
    \subfigure {\includegraphics[width=0.45\linewidth]{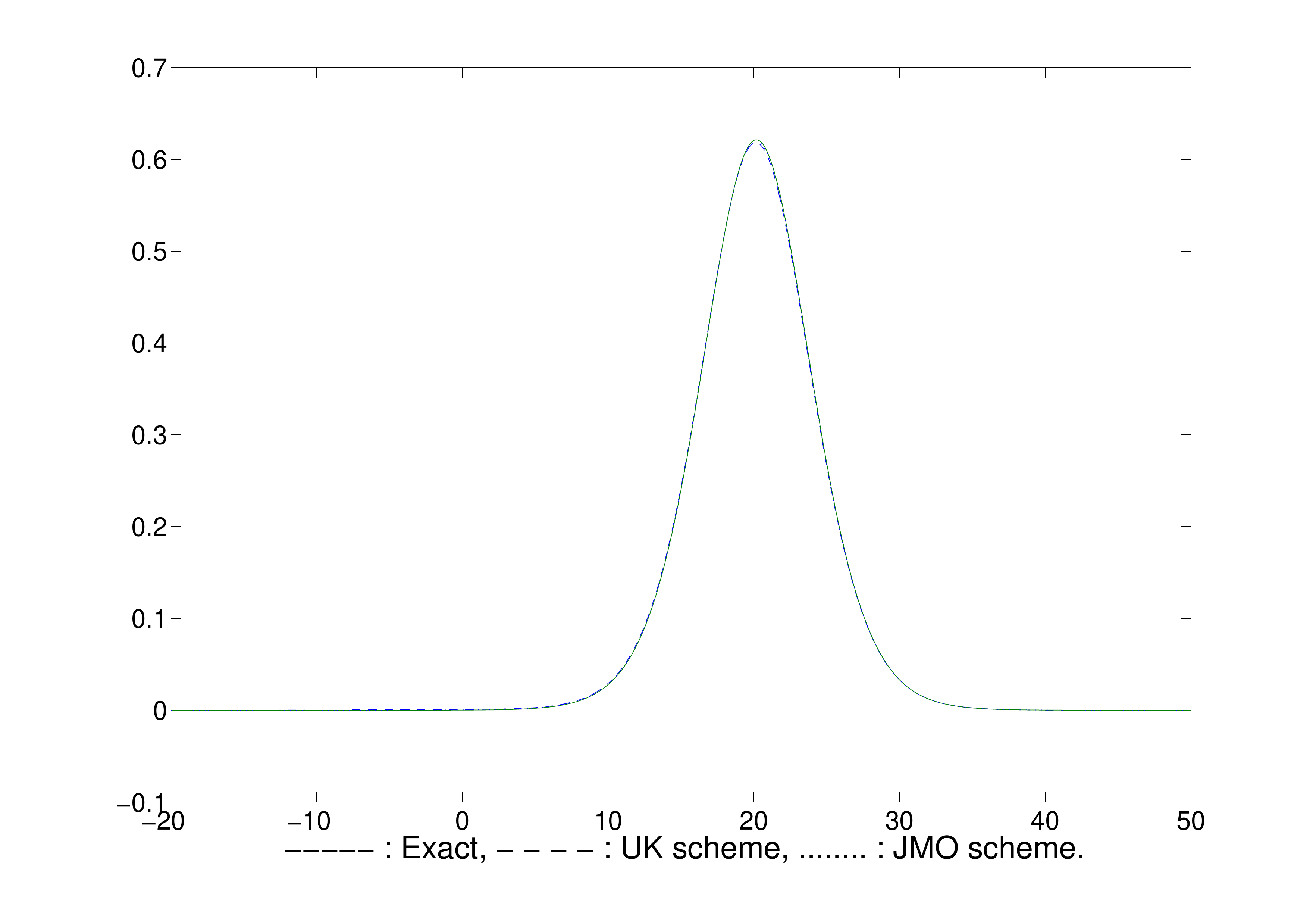}}
    \subfigure {\includegraphics[width=0.45\linewidth]{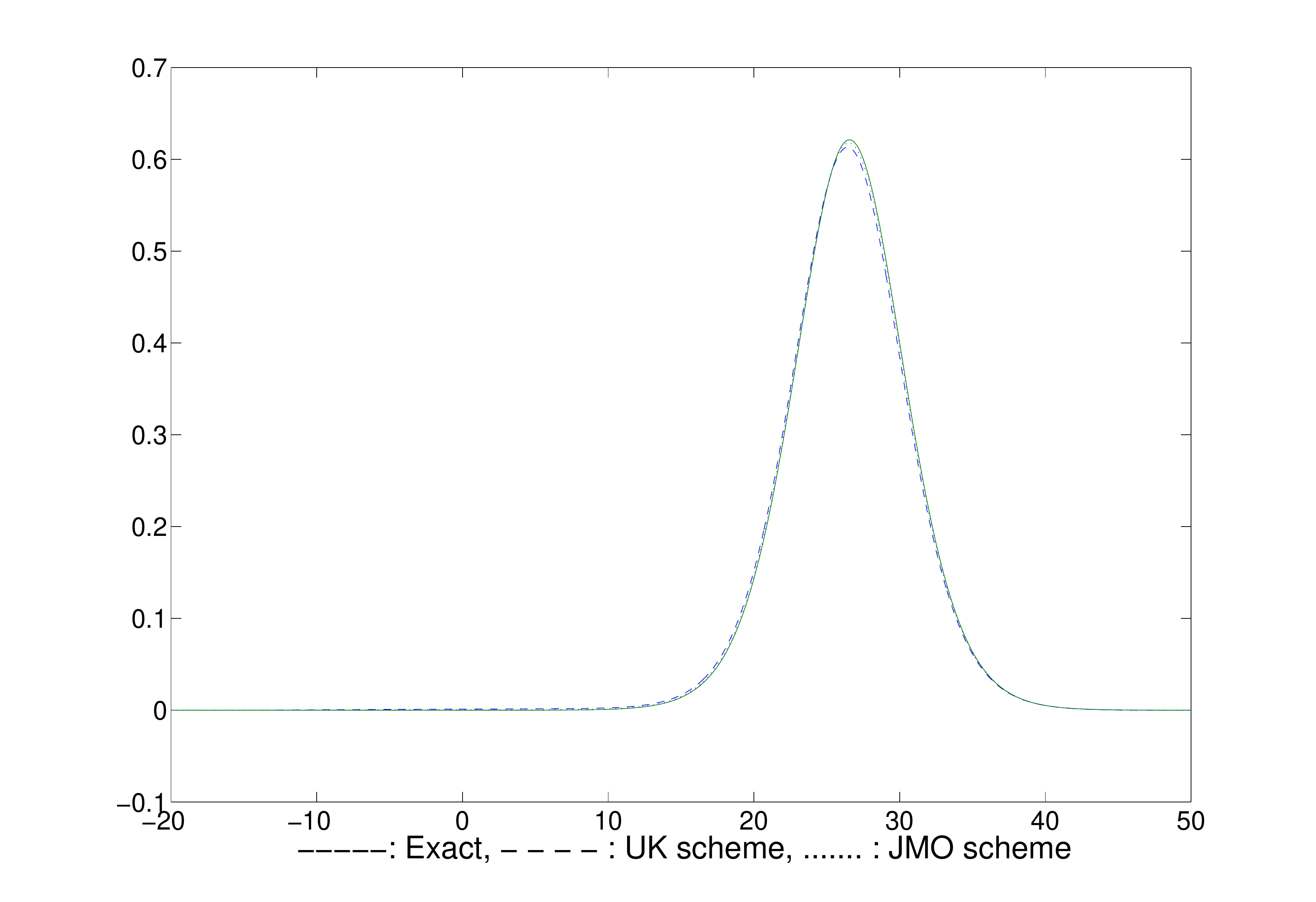}}
\caption{ Top Left: Exact and Numerical solution at time $t=30$ for both UK scheme and JMO scheme;
          Top Right: Exact and Numerical solution at time $t=60$ for both UK scheme and JMO scheme;
          Bottom Left: Exact and Numerical solution at time $t=90$ for both UK scheme and JMO scheme;
          Bottom Right: Exact and Numerical solution at time $t=120$ for both UK scheme and JMO scheme.}
\protect \label{fig:1}
\end{figure}

\subsection{Numerical experiment  2}

Now we will move on to the second numerical experiment. Here instead of one soliton we have considered two solitons, given by the initial data
 \begin{align*}
u(x,0) = \frac{105}{169} \left \lbrace sech^4 \left (  \frac{1}{2 \sqrt 13} ( x - 20)\right ) + \frac{1}{4} sech^4 \left (  \frac{1}{\sqrt 13} ( x - 60)\right ) \right \rbrace .
\end{align*}

This will essentially corresponds to the superposition of two solitons with different speeds given by the nonlinear term $uu_x$ of equation \eqref{eq:kawahara}. We have used $10,000$ points in space in the interval $[-100,100]$. Also we have run both schemes upto time $t = 50$. The following plot shows the behavior of the numerical solutions, where we see oscillatory structures of solitons.

\begin{figure}[htbp]
  \centering
    \subfigure {\includegraphics[width=0.45\linewidth]{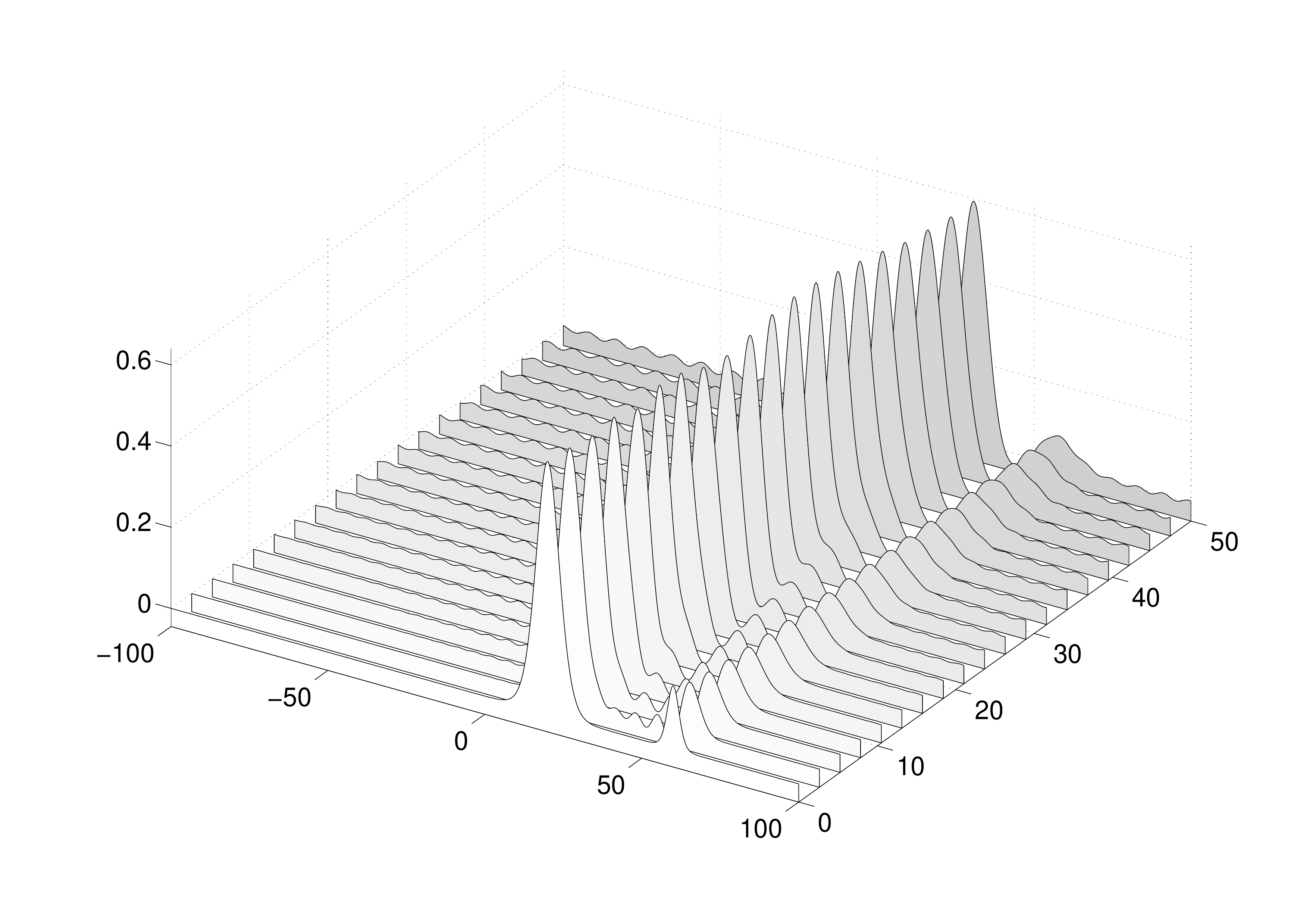}}
    \subfigure {\includegraphics[width=0.45\linewidth]{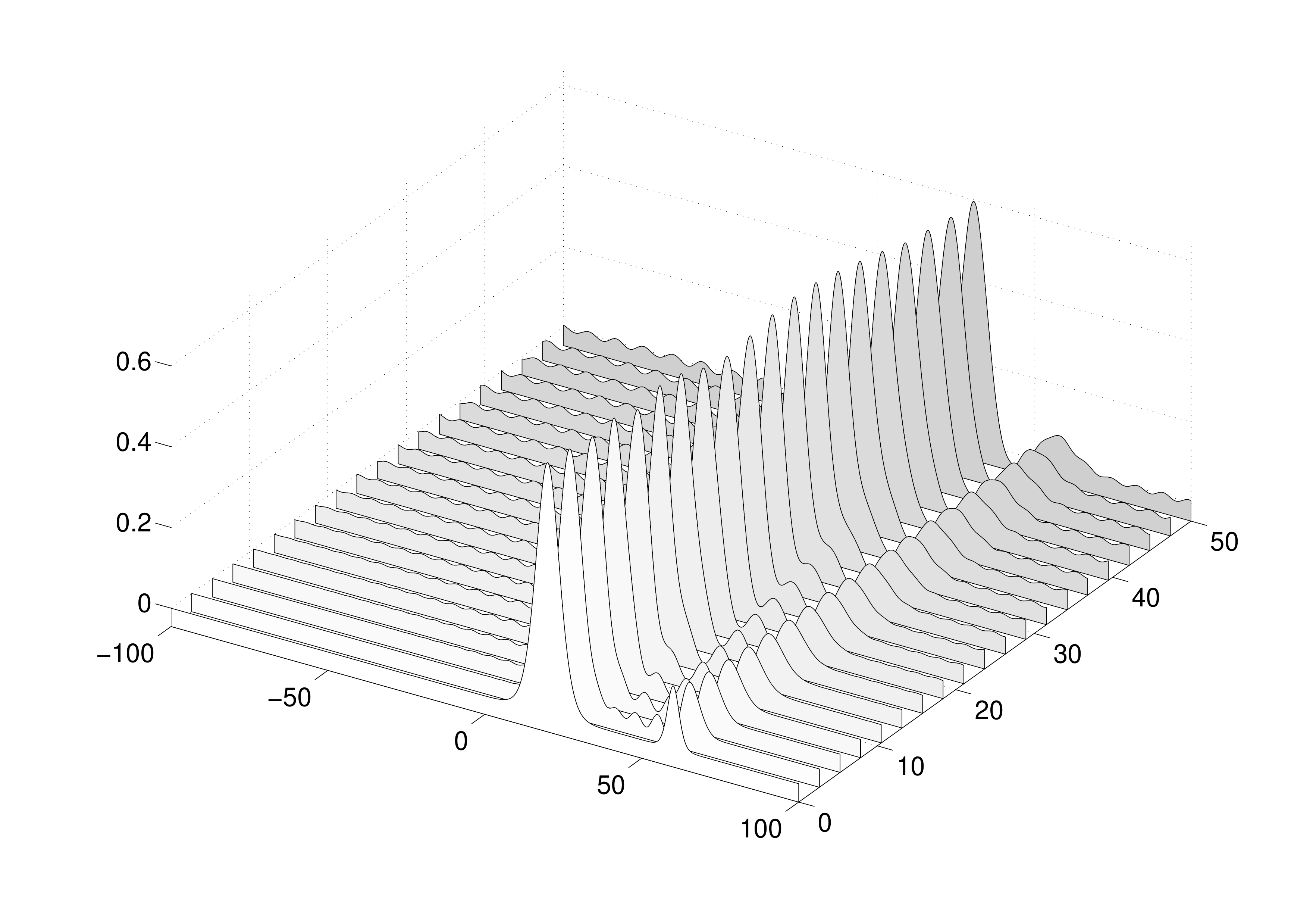}}
    \caption{ Left: Numerical solution for UK scheme;
              Right: Numerical solution for JMO scheme.}
\protect \label{fig:2}
\end{figure}

\thanks
 {Acknowledgements: The author would like to thank Nils Henrik Risebro for useful and motivating discussions.}

\end{document}